\documentclass[12pt]{article}
\usepackage{amsmath}
\usepackage{amssymb}

\usepackage{pstricks}
\usepackage{graphicx,psfrag}
\usepackage{amsthm}
\usepackage{color}
	
\newtheorem{theorem}{Theorem}

\newtheorem{lemma}{Lemma}

\newtheorem{corollary}{Corollary}

      \newcommand {\gam } {\gamma}          
                
              \newcommand {\ve}   {\varepsilon}

      \newcommand {\pl}   {\partial}

      \newcommand {\RRR}  {{\mathbb R}}

     \newcommand {\beq}  {\begin{equation}}
      \newcommand {\eeq}  {\end{equation}}
     \newcommand {\beqo}  {\begin{equation*}}
      \newcommand {\eeqo}  {\end{equation*}}

      \newtheorem{zam}{Remark}
      \newtheorem{opr}{Definition}

\newcommand{\R}{{\mathbb{R}}}

\title{Invisibility in billiards is impossible in an infinite number of directions}
\author{Alexander Plakhov\thanks{Department of Mathematics, University of Aveiro, Portugal, plakhov@ua.pt} \and Vera Roshchina\thanks{Collaborative Research Network, University of Ballarat, Australia, vroshchina@ballarat.edu.au}}

\begin{document}
\maketitle

\begin{abstract}
We consider the billiard in the exterior of a piecewise smooth  body in two-dimensional Euclidean space and show that
the maximum number of directions of invisibility in such billiard is at most finite.
\end{abstract}

\begin{quote}
\end{quote}

\begin{quote}
\end{quote}

\section{Introduction}

The main purpose of this work is to prove that billiard invisibility is impossible in an infinite number of directions in a class of bounded two-dimensional bodies with a piecewise smooth boundary. We consider such a body (set) and the billiard in the complement of this set. Invisibility in a direction means that almost all billiard particles that initially move in this direction and hit the body, after several reflections from the body's boundary are eventually redirected to the same trajectory, so the initial and final (infinite) intervals of this trajectory lie on the same straight line.

Even though research on invisibility is flourishing, and impressive progress has been made in the design of metamaterials (artificial materials engineered to bend electromagnetic waves around a concealed object), until very recently invisibility in mirror optics was largely overlooked both in engineering and mathematics. We refer the reader to our paper \cite{PlakhovRoshchinaNonlin2011} for a brief overview of recent developments in the field and historical background. The focus of this note is solely on the mathematical aspects of billiard invisibility.

We showed earlier that it is impossible to construct a body invisible in {\it all} directions.
This work serves to lower the known upper bound on the number of directions of invisibility: in the two-dimensional case for piecewise smooth bodies it is now reduced from `less than all' to `at most finite number' of directions.

Earlier there have been constructed 2D and 3D bodies invisible in 1 direction  \cite{0-resist} and from 1 point \cite{PlakhovRoshchinaOnePoint}. It is unknown if these results can be improved. There have been also designed infinitely connected sets invisible in 2 directions (in 2D case) and in 3 directions (in 3D case); notice however that these sets do not have piecewise smooth boundary and therefore do not satisfy the assumptions we impose in this paper.

In a somewhat different development Burdzy and Kulczycki \cite{BurdzyKulczycki2012} showed that it is possible to construct a body that is almost invisible in almost all directions with a given arbitrarily small accuracy. The body they constructed is the finite union of disjoint line segments contained in the unit circle, and its projection on most directions is close to the corresponding projection of the circle. Here `most' and `close' mean: up to a set of arbitrarily small measure.

Our paper is organized as follows. In Section~\ref{sec:main} we go over basic definitions and state our main result, then prove it in Section~\ref{sec:proof}.

\section{Notation and the main result}\label{sec:main}

\begin{opr}\label{o piecewise smooth}\rm
A set $\gam \subset \R^2$ is called a {\it piecewise smooth 1D set}, if it is the union of a finite number of curves,
$$
\gam = \cup_{i=1}^m \gam_i([a_i,b_i]),
$$
satisfying the following conditions: each curve is smooth and has a non self-intersecting interior, and the interiors of different curves are disjoint. That is, each function $\gam_i : [a_i,b_i] \to \R^2$ is of class $C^1$ (this implies that there exist the lateral derivatives of $\gam_i$ at $a_i$ and $b_i$); $|\gam_i'(t)| \ne 0$ for all $t \in [a_i,\, b_i]$;\, $\gam_i(t_1) = \gam_i(t_2)$ for $t_1,\, t_2 \in (a_i,\, b_i)$ implies $t_1 = t_2$; and for $i \ne j$,\, $\gam_i((a_i,\, b_i)) \cap \gam_j((a_j,\, b_j)) = \emptyset$.

The endpoints of the curves $\gam_i(a_i)$,\, $\gam_i(b_i)$,\, $i = 1,\ldots,m$ are called {\it singular points of} $\gam$, and the rest of the points of $\gam$ are called {\it regular}.
\end{opr}

\begin{opr}\label{o body}\rm
A compact set $B \subset \R^2$, whose boundary $\pl B$ is a piecewise smooth 1D set, is called {\it a body}. {\it Regular} and {\it singular} points of the body's boundary are determined in agreement with definition \ref{o piecewise smooth}. An example of a body see in fig. \ref{fig body}.
\end{opr}

\begin{figure}[h]
\centering
\includegraphics[keepaspectratio, height=120pt]{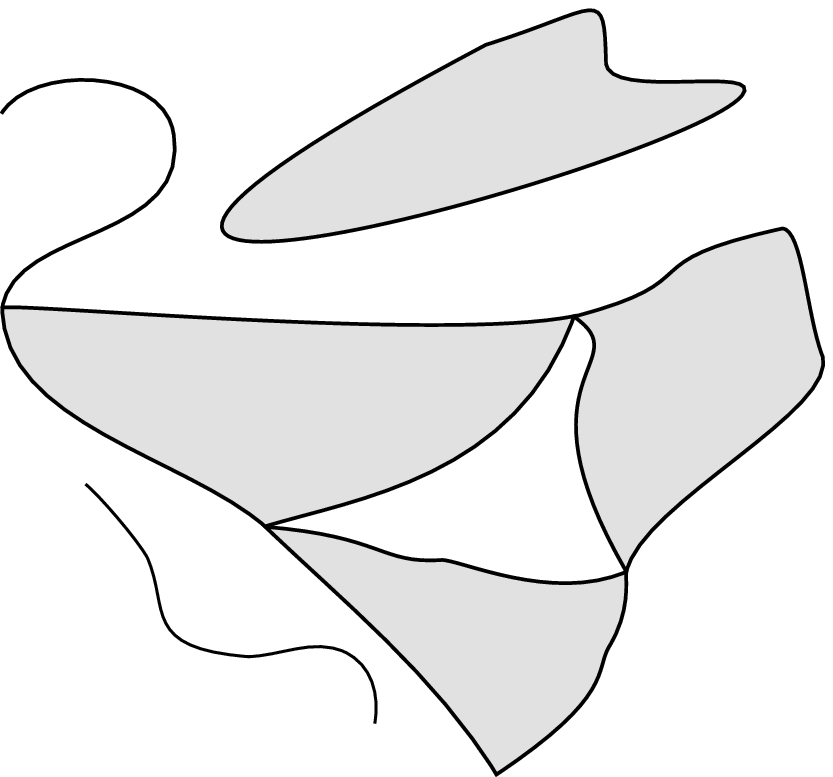}
\caption{A body.}
\label{fig body}
\end{figure}

\begin{zam}\label{r body}\rm
Note that according to definition \ref{o body}, a body is {\it not necessarily connected}.
\end{zam}

We consider the billiard in $\R^2 \setminus B$, where $B$ is a body.

\begin{opr}\label{o regular motion}\rm
A billiard motion $x(t),\, x'(t)$ is called {\it regular}, if it is defined for all $t \in \R$ and has a finite number of reflections at regular points of $\pl B$.
\end{opr}

According to this definition, the function $x(t)$ describing a regular billiard motion is piecewise linear, and its graph has a finite number of linear segments, with the initial and final segments being unbounded:
$$
x(t) = x + vt,\, \ t \le t_i; \quad x(t) = x^+ + v^+ t,\, \ t \ge t_f.
$$
Here $t_i$ and $t_f$ indicate the instants of the first and the final reflection of the billiard particle. Besides, the velocity $x'(t)$ of the particle is a piecewise constant function taking values in $S^1$.

The values $x,\, v$ are called {\it the initial data} and $x^+$,\, $v^+$ {\it the final data} of the motion. The final data are functions of the initial ones, $x^+ = x^+(x,v)$,\, $v^+ = v^+(x,v)$. These functions are continuous and defined on an open subset of $\R^2 \times S^1$. Each regular billiard motion is uniquely defined by its initial data (and also by its final data).

\begin{zam}\label{r initial data}\rm
If $x_1 - x_2$ is parallel to $v$ then the two motions with the initial data $x_1,\, v$ and $x_2,\, v$ can be obtained one from the other by a shift along $t$ (and therefore the corresponding trajectories $\{ x_1(t),\, t \in \R \}$ and $\{ x_2(t),\, t \in \R \}$ coincide). Therefore the billiard motion is well defined even if $x$ is contained in $B$ (see fig. \ref{fig inidata}). In this case the particle initially moves along the `negative' half-line with the direction vector $v$ until a collision with $B$.
\end{zam}

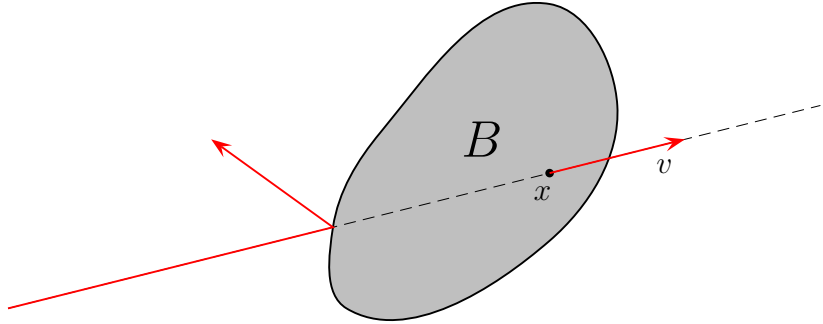
\begin{figure}[h]
\begin{picture}(0,120)
 \rput(7,0.2){
 \scalebox{0.9}{
\psecurve[fillstyle=solid,fillcolor=lightgray](3,1)(0,0)(-0.2,1.2)(0.5,2.6)(3,4.5)(4,3)(3,1)(0,0)(-0.2,1.2)
\rput(2,2.5){\scalebox{1.8}{$B$}}
\psline[linestyle=dashed,linewidth=0.4pt](-5,0)(7,3)
\psdot(3,2)
\rput(2.9,1.7){\scalebox{1.1}{$x$}}
\rput(4.7,2.1){\scalebox{1.1}{$v$}}
\psline[linecolor=red,arrows=->,arrowscale=2](3,2)(5,2.5)
\psline[linecolor=red,arrows=->,arrowscale=2](-5,0)(-0.2,1.2)(-2,2.5)
  }}
  \end{picture}
\caption{The motion with the initial data $x,\, v$.}
\label{fig inidata}
\end{figure}

\begin{opr}\label{o unperturbed}\rm
A regular billiard motion is {\it not perturbed} by the body $B$ (or just {\it unperturbed}), if for its initial $x,\, v$ and final $x^+,\, v^+$ data we have $v^+ = v$ and $x^+ - x$ is parallel to $v$.
\end{opr}

\begin{zam}\label{r unperturbed}\rm
It follows from this definition that the (unbounded) initial and final linear segments of the trajectory corresponding to an unperturbed motion lie on a single straight line.
\end{zam}

Let $v \in S^1$.

\begin{opr}\label{o invisible}\rm
A body $B$ is said to be {\it invisible in the direction $v$}, if for almost all $x \in \R^2$ the billiard motion with the initial data $x, v$ is regular and not perturbed by $B$. The vector $v$ is called {\it a direction of invisibility} for $B$.
\end{opr}

\begin{zam}\label{r invis1}\rm
If $B$ is invisible in a direction $v$, then it is also invisible in the direction $-v$.
\end{zam}

\begin{zam}\label{r invis2}\rm
For all invisible bodies we know, each billiard motion is either unperturbed, or hits the body's boundary at a singular point.
\end{zam}

Introduce some notation. The vector obtained by counterclockwise rotation of $v$ by $\pi/2$ is denoted by $v^\perp$. Thus we have $(v^\perp)^\perp = -v$,\, $\langle u,\, v \rangle = \langle u^\perp,\, v^\perp \rangle$, and $\langle u,\, v^\perp \rangle = -\langle u^\perp,\, v \rangle$. Here and in what follows $\langle \cdot\,, \cdot \rangle$ denotes the scalar product in $\RRR^2$.

Let $\xi \in \pl B$ be a regular point; then $n(\xi)$ denotes the outer unit normal to $B$ at $\xi$. Conv$\,B$ denotes the convex hull of $B$. The set $\pl B \cap \pl(\text{Conv}\,B)$ is called the {\it convex part} of the boundary of $B$.

The main theorem is as follows.

\begin{theorem}\label{thm:main}
Bodies invisible in infinitely many directions do not exist.
\end{theorem}

Its proof is given in the next section.

\section{Proof of the main theorem}\label{sec:proof}

To prove theorem~\ref{thm:main} we need several technical results.

\begin{lemma}\label{l0 regtraj}
If there exists a regular billiard motion with initial data $x,\, v$ such that $v^+(x,v) \ne v$, then $B$ is not invisible in the direction $v$.
\end{lemma}

\begin{proof}
Since the function $v^+$ is continuous, for $\check{x}$ in a small neighborhood of $x$ we have $v^+(\check{x},v) \ne v$. Therefore $B$ is not invisible in the direction $v$.
\end{proof}

\begin{lemma}\label{l1 sing}
Let $B$ be invisible in two linearly independent directions $v_1$ and $v_2$. Then each point of $\pl B \cap \pl(\text{Conv}\,B)$ is a singular point of $\pl B$.
\end{lemma}

\begin{proof}
Suppose that $\xi \in \pl B \cap \pl(\text{Conv}\,B)$ is a regular point of $\pl B$. Then it is also a regular point of $\pl(\text{Conv}\,B)$. At least one of the values $\langle v_1,\, n(\xi) \rangle$,\, $\langle v_2,\, n(\xi) \rangle$ is nonzero. Without loss of generality assume that $\langle v_1,\, n(\xi) \rangle \ne 0$. Replacing if necessary $v_1$ with $-v_1$, we can ensure that $\langle v_1,\, n(\xi) \rangle < 0$.

The billiard motion in  $\R^2 \setminus \pl(\text{Conv}\,B)$ with the initial data $\xi,\, v_1$ is regular and is described by the function
$$
x(t) = \left\{ \begin{array}{cc} \xi + v_1 t, & \text{if } \ t \le 0\\ \xi + v_1^+ t, & \text{if } \ t \ge 0 \end{array} \right.,
$$
where $v_1^+ = v_1 - \langle v_1,\, n(\xi) \rangle\, n(\xi)$ is defined according to the law of elastic reflection (see fig. \ref{fig conv}). Obviously, $v_1^+ \ne v_1$. Note that the same motion is also a regular motion in $\R^2 \setminus B$, and so by lemma \ref{l0 regtraj}, $B$ is not invisible in the direction $v_1$.
\end{proof}

\begin{figure}[h]
  \centering
\includegraphics[keepaspectratio, height=120pt]{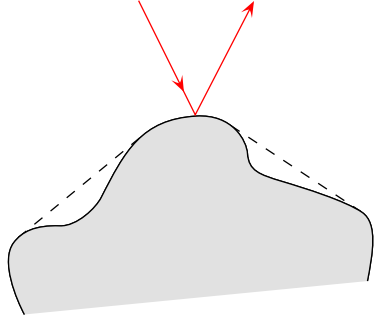}
\caption{Reflection from the convex part of the body's boundary.}
\label{fig conv}
\end{figure}

\begin{corollary}\label{cor1}
Under the assumptions of lemma \ref{l1 sing}, $\text{Conv}\,B$ is a polygon.
\end{corollary}

\begin{proof}
Indeed, all the extreme points of $\text{Conv}\,B$ belong to $\pl B \cap \pl(\text{Conv}\,B)$; therefore they are singular points of $\pl B$. It follows that the set of extreme points is finite, therefore $\text{Conv}\,B$ is a polygon.
\end{proof}

\begin{opr}\label{o supp line}\rm
Let
$$
a_n = \sup_{x\in B}\, \langle x,\, n \rangle.
$$
The line $\langle x,\, n \rangle = a_n$ is called the {\it $n$-supporting line} for the body $B$. The set $B \cap \{ x : \langle x,\, n \rangle = a_n \}$ is called the {\it $n$-maximal set} of $B$, and points of this set are called {\it $n$-maximal points} of $B$.
\end{opr}

\begin{lemma}\label{l2 2points}
Under the assumptions of lemma \ref{l1 sing}, the $v_1^\perp$-maximal set of $B$ contains at least two points.
\end{lemma}

\begin{proof}
For the sake of simplicity write $v$ in place of $v_1$, omitting the subscript. Since $B$ is compact, the $v_1^\perp$-maximal set is not empty. It remains to prove that it is not a singleton. Assume the contrary, that is,
$$
B \cap \{ x : \langle x,\, v^\perp \rangle = a_{v^\perp} \} = \{ \xi \}.
$$
Consider the curves $\gam_i$ comprising $\pl B$ that contain $\xi$ (see fig. \ref{fig single}). Clearly, $\xi$ is an endpoint of each such curve. Assume without loss of generality that $\xi$ is the second endpoint of each curve, $\xi = \gam_i(b_i)$, and therefore
$$
\langle \gam'_i(b_i),\, v^\perp \rangle \ge 0.
$$
(Otherwise we just reparameterize the curve by $\tilde\gam_i(t) = \gam_i(-t)$,\, $t \in [-b_i,\, -a_i]$.)

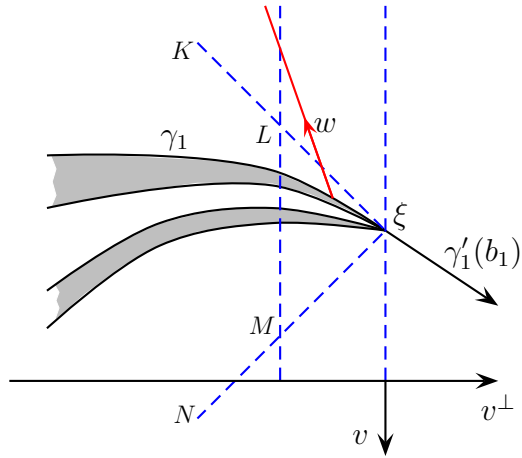
\begin{figure}[h]
\begin{picture}(0,190)
 \rput(4.8,1.5){
 \scalebox{1}{
 \pspolygon[linewidth=0,linecolor=white,fillstyle=solid,fillcolor=lightgray]
 (0.55,3)(1.5,3)(2.5,2.95)(3,2.91)(3.5,2.8)(4,2.59)(4.5,2.32)
 (5,2)(4.5,2.27)(4,2.48)(3.5,2.62)(2.5,2.63)(1.5,2.49)(0.55,2.3)
 (0.62,2.47)(0.57,2.65)(0.62,2.82)
  \pspolygon[linewidth=0,linecolor=white,fillstyle=solid,fillcolor=lightgray]
  (0.6,1.28)(1,1.57)(1.5,1.87)(2,2.1)(2.5,2.23)(3,2.3)(3.5,2.3)(4,2.24)(4.5,2.13)
  (5,2)(4.5,2.05)(4,2.1)(3.5,2.1)(3,2.05)(2.5,1.98)(2,1.8)(1.5,1.52)(1,1.15)(0.6,0.8)
  (0.66,0.97)(0.62,1.07)(0.66,1.18)
\pscurve(0.5,3)(3.5,2.8)(5,2)
\pscurve(0.5,2.3)(3.5,2.6)(5,2)
\pscurve(0.5,1.2)(2,2.1)(3.5,2.3)(5,2)
\pscurve(0.5,0.7)(2,1.8)(3.5,2.1)(5,2)
\psline[arrows=->,arrowscale=2](5,2)(6.5,1)
\psline[linecolor=blue,linestyle=dashed](5,0)(5,4.99)
\psline[linecolor=blue,linestyle=dashed](2.5,4.5)(5,2)(2.5,-0.5)
\psline[linecolor=blue,linestyle=dashed](3.6,0)(3.6,4.99)
\psline[linecolor=red](4.3,2.43)(3.4,4.99)
\rput(2.3,4.4){\scalebox{0.8}{$K$}}
\rput(3.38,3.28){\scalebox{0.8}{$L$}}
\rput(3.36,0.74){\scalebox{0.8}{$M$}}
\rput(2.33,-0.44){\scalebox{0.8}{$N$}}
\psline[arrows=->,arrowscale=2](0,0)(6.5,0)
\psline[arrows=->,arrowscale=2](5,0)(5,-1)
\rput(2.2,3.2){$\gam_1$}
\rput(5.2,2.2){$\xi$}
\rput(6.5,-0.35){$v^\perp$}
\rput(4.67,-0.75){$v$}
\rput(6.3,1.7){$\gam_1'(b_1)$} 
\psline[arrows=->,arrowscale=1.7,linecolor=red](4.3,2.43)(3.92,3.51)
\rput(4.2,3.4){$w$}
  }}
  \end{picture}
\caption{The $v_1^\perp$-maximal set is a singleton.}
\label{fig single}
\end{figure}

Recall that by corollary \ref{cor1}, Conv$\,B$ is a polygon, therefore $B$ is contained in an angle with the vertex at $\xi$, which is (except for the vertex) contained in the half-plane $\langle x,\, v^\perp \rangle < 0$. This fact allows one to sharpen the previous inequality,
$$
\langle \gam'_i(b_i),\, v^\perp \rangle > 0.
$$
This also implies that in the coordinate system with the $x$-axis parallel to $v^\perp$ and the $y$-axis parallel to $-v$, the curves $\gam_i$ containing $\xi$ can be locally (in a small strip $a_{v^\perp}-\ve < \langle x,\, v^\perp \rangle < a_{v^\perp}$) represented by graphs of functions. Moreover, these functions form a finite monotone sequence. Let the graph of the largest function represent the curve $\gam_1$.

One can choose $\ve > 0$ in such a way that the aforementioned angle (and therefore $B$) is contained in the angle
$$
A_\ve = \Big\{ x : \ \Big\langle \frac{x-\xi}{|x-\xi|}, \ v^\perp \Big\rangle \le 2\ve \Big\}.
$$
(the angle $\measuredangle K\xi N$ in fig. \ref{fig single}).

We have
$$
\text{either \, (i) } \max_i\, \langle \, \gam'_i(b_i),\, v \rangle \ge 0, \quad \text{ or \, (ii) } \min_i\, \langle \gam'_i(b_i),\, v \rangle \le 0
$$
(with the maximum and minimum taken over all $i$ such that $\gam_i(b_i) = \xi$). Assume that (i) holds (the case (ii) is considered in a similar way). By our convention, the maximum is achieved for $\gam_1$, therefore
$$
\langle \gam'_1(b_1),\, v \rangle = \max_i\, \langle \gam'_i(b_i),\, v \rangle \ge 0.
$$
Then for all $x$ with $a_{v^\perp} - \langle x,\, v^\perp \rangle$ positive and sufficiently small, the motion with the initial data $x,\, v$ has (within the strip) a single reflection at a point $x_1 = x_1(x)$ of the curve $\gam_1$. Besides, the velocity $v^+(x)$ after the reflection is close to
$$
v^+ = v - 2\langle v,\, \gam_1'(b_1)^\perp \rangle\, \gam_1'(b_1)^\perp,
$$
which means that $|v^+(x) - v^+| < \ve$.

One easily sees that $\langle v^+,\, v^\perp \rangle \ge 0$; indeed,
$$
\langle v^+,\, v^\perp \rangle = - 2\langle v,\, \gam_1'(b_1)^\perp \rangle\, \langle \gam_1'(b_1)^\perp,\, v^\perp \rangle = 2\langle v^\perp,\, \gam_1'(b_1) \rangle\, \langle \gam_1'(b_1),\, v \rangle \ge 0.
$$
It follows that $\langle v^+(x),\, v^\perp \rangle > -\ve$.

If $a_{v^\perp} - \langle x,\, v^\perp \rangle$ is positive and sufficiently small then each ray $x_1 + wt,\ t \ge 0$ with the vertex at $x_1 = x_1(x)$ and $|w| = 1$,\, $\langle w, v^\perp \rangle > -\ve$ is entirely contained in the set
$$
V_\ve = \Big\{ x : \ \langle x,\, v^\perp \rangle > a_{v^\perp}-\ve, \ \ \Big\langle \frac{x-\xi}{|x-\xi|}, \ v^\perp \Big\rangle > 2\ve \Big\}.
$$
(the set bounded by and to the right of the broken line $KLMN$ in fig. \ref{fig single}). This implies that each particle with the corresponding initial data $x,\, v$ makes a single reflection from $\gam_1$ and then moves freely forever. Indeed, the further motion is in $V_\ve$, which is the union of the strip $a_{v^\perp}-\ve < \langle x,\, v^\perp \rangle < a_{v^\perp}$ and the set $\R^2 \setminus A_\ve$ (exterior of the angle). There are no further reflections inside the strip, and no point of reflection can belong to $\R^2 \setminus A_\ve$ (since $B \subset A_\ve$).

It remains to note that $\langle v,\, \gam_1'(b_1)^\perp \rangle < 0$ and therefore $v^+(x) \ne v$. We come to a contradiction with lemma \ref{l0 regtraj}, which states that $B$ is not invisible in the direction $v$.
\end{proof}

The following corollary of lemma \ref{l2 2points} is obvious.

\begin{corollary}\label{cor2}
Under the assumptions of lemma \ref{l1 sing}, $v_1$ is parallel to a side of the polygon Conv$\,B$.
\end{corollary}

Let us now prove the theorem. Let a body $B$ be invisible in several (at least two) directions. By corollary \ref{cor1}, Conv$\,B$ is an $m$-gon, $m \ge 3$. By corollary \ref{cor2}, each direction of invisibility is parallel to a side of this $m$-gon. Therefore there are at most $m$ directions of invisibility. The theorem is proved.

\section*{Acknowledgements}

The work was supported by the FCT research project PTDC/MAT/113470/2009 and by the Collaborative Research Network, University of Ballarat, Australia. In addition, the first author has been partly supported by {\it FEDER} funds through {\it COMPETE}--Operational Programme Factors of Competitiveness and by Portuguese funds through the {\it Center for Research and Development in Mathematics and Applications} (CIDMA) and the Portuguese Foundation for Science and Technology (FCT), within project PEst-C/MAT/UI4106/2011 with COMPETE number FCOMP-01-0124-FEDER-022690.

\end{document}